\documentclass[11pt, reqno]{amsart}

\usepackage[noadjust]{cite}
\usepackage[letterpaper, hmarginratio=1:1]{geometry}
\usepackage{amsmath, amsthm, amssymb, amsfonts, enumerate, bbm, comment}
\usepackage{hyperref}

\numberwithin{equation}{section}
\theoremstyle{plain}
 \newtheorem{theorem}{Theorem}[section]

 \newtheorem{proposition}[theorem]{Proposition}
 
 \newtheorem{lemma}[theorem]{Lemma}
 \newtheorem{corollary}[theorem]{Corollary}

\newtheorem{remark}[theorem]{Remark}

\theoremstyle{definition}
\newtheorem{definition}[theorem]{Definition}

\def \N {\mathbb{N}}
\def \R {\mathbb{R}}

\def \Z {\mathbb{Z}}
\def \E {\mathbb{E}}
\def \P {\mathbb{P}}
\def \S {\mathbb{S}}
\def \one {{\bf 1}}

\def \LL {\mathcal{L}}

\def \< {\langle}
\def \> {\rangle}

\renewcommand\epsilon{\varepsilon}

\renewcommand\P{{\mathbb{P}}}

\newcommand\eps{{\varepsilon}}

\DeclareMathOperator{\supp}{supp}
\DeclareMathOperator{\Comp}{Comp}
\DeclareMathOperator{\Incomp}{Incomp}

\DeclareMathOperator{\Dom}{Dom}

\newcommand{\IncompTwo}{ \Incomp(m , 2^{j-1} \rho, 2^j \rho)}
\newcommand{\Incompalphan}{ \Incomp(\alpha n , 2^{j-1} \rho, 2^j \rho)}

\def \Dhat {\widehat{D}}

\begin{document}

\title{Tail bounds for gaps between eigenvalues of sparse random matrices}

\author{Patrick Lopatto}
\thanks{P.L. is partially supported by the NSF Graduate Research Fellowship Program under grant DGE-1144152.}
\author{Kyle Luh}
\thanks{K. Luh was partially supported by NSF postdoctoral fellowship DMS-1702533.}

\maketitle

\begin{abstract}
We prove the first eigenvalue repulsion bound for sparse random matrices.  As a consequence, we show that these matrices have simple spectrum, improving the range of sparsity and error probability from work of the second author and Vu.  We also show that for sparse Erd\H{o}s--R\'enyi graphs, weak and strong nodal domains are the same, answering a question of Dekel, Lee, and Linial.  
\end{abstract}

{\setcounter{tocdepth}{1}
	\tableofcontents
}

\section{Introduction}
The gaps between eigenvalues of symmetric random matrices have been extensively studied by mathematicians and physicists. For the classical integrable ensembles, the Gaussian Orthogonal Ensemble and Gaussian Unitary Ensemble, the limiting spectral distribution follows the semicircle law.  For an individual eigenvalue gap, however, the limiting distribution was only recently obtained \cite{tao2013gap}.  Rapid progress in random matrix theory has permitted the extension of this result to a large class of random matrix models \cite{tao2012universality, erdos2015gapuniversality, shcherbina2014universality, bekerman2015universality, EKYY12,EKYY13,EPR10,ESY11,rigidity, TV10,TV11,aggarwal2018goe, aggarwal2019bulk,huang2015bulk, che2019local,che2019universality,che2019universality2,pillai2014universality,che2017universality}. 

Much effort has been expended on understanding the extremal eigenvalue gaps, in particular the largest eigenvalue gap in the bulk of the spectrum, $\delta_{\mathrm{max}}$.  Ben Arous and Bourgade \cite{benarous2013extremegaps} demonstrated that for the $n \times n$ GUE normalized so that its spectrum is supported on $[-2, 2]$, so that the typical inter-particle distance in the bulk is about $n^{-1}$, the largest bulk gap is of order $n^{-1} \sqrt{\log n} $.  Figalli and Guionnet extended this result to $\beta$-ensembles with $\beta = 2$ \cite{figalli2016transport}.  In \cite{feng2018large}, Feng and Wei showed that the fluctuations of the largest gap are of order  $n^{-1} \sqrt{\log n} $ and computed the limiting distribution.  
In work of the first author with Landon and Marcinek, the largest gap results of \cite{benarous2013extremegaps,feng2018large} were extended  to generalized Wigner matrices \cite{landon2018comparison}, including those with discrete entry distributions. We note that recent work of Bourgade \cite{bourgadeprep}, which presents a concise analysis of the convergence to equilibrium of Dyson Brownian motion, is able to recover the same result at the cost of imposing a weak smoothness assumption on the matrix entries.

While we now have a substantial understanding of the largest eigenvalue gap, the smallest gap, $\delta_{\mathrm{min}}$, is more difficult to investigate because it lies well below the typical inter-particle distance.  Bourgade and Ben Arous \cite{benarous2013extremegaps} showed using the determinantal structure of the GUE that its smallest gap is of order $n^{-4/3}$.  In \cite{Feng2018GOE}, Feng, Tian, and Wei identified the normalized limit of the smallest eigenvalue gap of the GOE and found that the gap is of order $n^{-3/2}$; their argument builds on techniques previously developed by Feng and Wei to study circular $\beta$-ensembles \cite{Feng2018small}. Currently, the smallest gap lies outside of the purview of traditional universality results such as the Four Moment Theorem \cite{tao2014fourmoment}, and the techniques in the recent work \cite{landon2018comparison} are not applicable. The strongest available result is in the recent work of Bourgade \cite{bourgadeprep}, which shows universality of the smallest gap, but requires that the matrix entries possess a weak form of smoothness. At present, no universality results exist for the smallest gap for matrices that are sparse or have discrete entry distributions, such as a matrix of Bernoulli random variables.

While tail bounds are known for the individual gaps when the matrix entries are more general random variables \cite{tao2012universality, TV10}, the error rates are not strong enough to take a union bound to conclude anything about the minimum gap.  We now scale the matrices so that their spectrum lies on $[-2\sqrt{n}, 2 \sqrt{n}]$, which makes the average inter-particle distance $n^{-1/2}$; we take this convention to match the existing tail bound literature, and it remains in force throughout the rest of the paper. For Hermitian matrices, under stringent smoothness and decay assumptions on the random variables, a result of Erd\H{o}s, Schlein, and Yau \cite{erdos2010wegnerestimate} implies that there exists a small constant $c >0$ such that
$$
\P\left(\delta_{\mathrm{min}} \leq \frac{\delta}{n^{1/2}} \right) = o(n \delta^3) + \exp(-c n)
$$
for any $\delta > 0$.  For discrete random variables, it was a milestone just to show that $\delta_{\mathrm{min}} > 0$ \cite{tao2017simplespectrum}.  In particular, Tao and Vu showed that for any $A > 0$, with probability at least $1 - n^{-A}$ a random symmetric matrix has simple spectrum, meaning every eigenvalue appears with multiplicity one.  In follow-up work with Nguyen \cite{nguyen2017gaps}, they showed the following tail bound for the eigenvalue gaps. Given eigenvalues $\lambda_i$ labeled in ascending order, we denote the gaps by $\delta_i= \lambda_{i+1} - \lambda_i$.
\begin{theorem}[{\cite[Theorem 2.1]{nguyen2017gaps}}]
There exists a constant $c>0$ such that the following holds for the eigenvalue gaps, $\delta_i$, of a real symmetric Wigner matrix.  For any $n^{-c} \leq \alpha \leq c$ and $\delta \geq n^{-c/\alpha}$, 
$$
\sup_{1 \leq i \leq n-1} \P\left(\delta_i \leq \frac{\delta}{n^{1/2}} \right) = O\left( \frac{\delta}{\alpha^{1/2}} \right).
$$
\end{theorem}
Setting $\alpha = n^{-c}$, one can deduce that a real symmetric random matrix has simple spectrum with probability at least $1 - O(\exp(-n^c))$. 
A related problem, posed by Babai, is whether the adjacency matrix of an Erd\H{o}s--R\'enyi random graph has simple spectrum.  This was resolved affirmatively for all dense random graphs in \cite{tao2017simplespectrum, nguyen2017gaps}.   A consequence in complexity theory is that for such random graphs the graph isomorphism problem is in complexity class $\mathcal{P}$ \cite{babai1982isomorphism}.  

In this work we study the eigenvalue gaps of sparse random matrices. The theory of sparse random matrices is of interest in its own right, but it also has innumerable applications in computer science and statistics.  In contexts where sparse random matrices have similar spectral guarantees as their dense counterparts, they offer significant advantages as they require less space to store, allow quicker multiplication, and need fewer random bits to generate \cite{ballard2013communication, bah2013construction, auyeung2017sparse, dasgupta2010JL, nelson2013OSNAP, clarkson2017lowrank}.  A popular model for such matrices is to consider the Hadamard (entrywise) product of a dense random matrix and a sparse matrix of independent (up to symmetry) indicator variables with expectation $p = p(n)$.  Much work has been done to transfer the results known for dense random matrices to the sparse setting \cite{basak2017sparse, basak2018sharp, bourgade2017sparse, lee2018tracywidom, huang2015bulk, litvak2012smallest, wood2012circular, rudelson2018sparse, basak2018sharp}.  Although the results resemble their dense analogues, the sparsity brings about a variety of complications in the proofs.  
Only recently, the second author and Vu showed that for a large class of random variables and for $p \geq n^{-1 + \eps}$ with $\eps >0$, a sparse random matrix has simple spectrum with probability at least $1 - O_{\eps}(\exp(-(np)^{1/128}))$ \cite{luh2018sparse}, where this notation indicates that the implied constant depends on $\eps$.  This implies that the graph isomorphism problem restricted to this class of sparse random graphs is in complexity class $\mathcal{P}$.

Our main contribution is to go beyond verifying such matrices have simple spectrum and prove a tail bound for the minimal eigenvalue gap of sparse random matrices with $p \ge C \log^{7 + \eps}(n)/n$. In comparison with \cite{luh2018sparse}, our results represent an improvement in both error probability and the range of sparsity considered. As an application of our tail bound, we show that for sparse Erd\H{o}s--R\'enyi graphs, weak and strong nodal domains are the same, answering a question of Dekel, Lee, and Linial \cite{dekel2011nodaldomains}. Our results also expand the range of sparse graphs for which the graph isomorphism problem is known to be in $\mathcal{P}$. Related to this last application is the graph matching problem, for which various algorithms contingent on simple spectrum are known \cite{umeyama1988eigendecomposition,lyzinski2016graph,aflalo2014graph}; our results similarly extend their range of applicability.\\

\noindent {\bf Acknowledgments.} The authors thank the anonymous referees for their detailed comments, which substantially improved the paper.

\section{Main Results}
We begin with a formal definition of our random matrix model.
\begin{definition} \label{def:matrixmodel}
We let $M_n$ denote a symmetric random matrix with entries 
$$
m_{ij} = \xi_{ij} \chi_{ij},
$$
where the $\xi_{ij}$ are independent (for $i\ge j$), mean zero, variance one, and subgaussian with subgaussian moment $B$, and the $\chi_{ij}$ are independent  (for $i\ge j$) Bernoulli random variables with $\E \chi_{ij} = p$.
\end{definition}

\begin{theorem}\label{thm:main}
Let $M_n$ be as in Definition~\ref{def:matrixmodel}, and fix $\nu >0$.  There exist constants $C_{\ref{thm:main}}, c_{\ref{thm:main}}, c'_{\ref{thm:main}} > 0$, depending only on the subgaussian moment $B$, such that for
$$\frac{C_{\ref{thm:main}}\log^{7 + \nu} n}{n} \le p \le 1 $$
 and 
$$(np)^{-1/(7 + \nu)}  \leq \alpha \leq \frac{c'_{\ref{thm:main}}}{\log n},$$
the following holds for the gaps between the eigenvalues, $\delta_i= \lambda_{i+1} - \lambda_i$.  For any $\delta \geq \exp(-\alpha^{-1})$, 
$$
\sup_{1 \leq i \leq n-1} \P\left(\delta_i \leq \delta \exp\left(-c_{\ref{thm:main}} \frac{\log(1/p)}{\log np}\right)    \sqrt{\frac{p}{n}} \right) \leq     C_{\ref{thm:main}} \frac{\delta}{\alpha}.
$$

\end{theorem}
Observe that there is a trade-off in the strength of the error bound and the size of the eigenvalue gap, determined by the value of $\alpha$.  For example, 
if we choose $\alpha = c_{\ref{thm:main}}/ \log n$, we obtain the following result.
\begin{corollary} \label{cor:largegap}
Let $M_n$ be as in Definition~\ref{def:matrixmodel}, and fix $\nu >0$.  There exist $C_{\ref{cor:largegap}}, C_{\ref{cor:largegap}}' > 1$, such that for
$p \geq \frac{C_{\ref{thm:main}}\log^{7 + \nu} n}{n},$
$$
\sup_{1 \leq i \leq n-1} \P\left(\delta_{i} \leq    \delta \exp\left(-c_{\ref{thm:main}} \frac{\log(1/p)}{\log np}\right)   \sqrt{\frac{p}{n}} \right) \le C_{\ref{cor:largegap}} \delta \log n.
$$
for $\delta \geq n^{-C_{\ref{cor:largegap}}'}$.  By a union bound, 
$$
 \P\left(\delta_{\mathrm{min}} \leq    \frac{\sqrt{p}}{n^{3/2 + o(1)}} \right) = o(1).
$$ 
\end{corollary}

At the other extreme, setting $\alpha = (np)^{-1/(7 + \nu)}$ and $\delta = \exp( - \alpha^{-1})$, we have the following result.
\begin{corollary}
Let $M_n$ be as in Definition \ref{def:matrixmodel}, and fix $\nu>0$.  For
$p \geq \frac{C_{\ref{thm:main}}\log^{7 + \nu} n}{n},$
$$
\P(M_n \emph{ has eigenvalues with multiplicity}) \leq \exp\Big(-  \frac{1}{2}(np)^{1/(7 + \nu) }  \Big).
$$
\end{corollary}
Observe that when $p=1$, which is the dense case considered in \cite{nguyen2017gaps}, the above two corollaries recover \cite[Corollary 2.2]{nguyen2017gaps} and \cite[Corollary 2.3]{nguyen2017gaps}, which are the analogous extreme cases of the bound in \cite[Theorem 2.1] {nguyen2017gaps}.

\begin{remark}\label{r:optimal}
This result improves the range of sparsity in \cite{luh2018sparse} from $n^{-1+\eps}$ for some $\eps >0$ to $\log n^{7 + \nu}/n$.  
Even in the regime $p \geq n^{-1+\eps}$, our result improves on the bound in \cite{luh2018sparse} where the probability of not having a simple spectrum was less than $\exp(- (np)^{1/124})$.  However, we suspect that the optimal bound should be $\exp(-c np)$ for some constant $c > 0$.  The sparsity range of Theorem \ref{thm:main} is near optimal as $p = o(\log n/n)$ yields multiple rows and columns entirely of zeros.  This generates repeated eigenvalues at 0. 
\end{remark}

We also have the same result for adjacency matrices of random Erd\H{o}s--R\'enyi graphs. Let $G(n,p)$ denote the random graph on $n$ vertices with edges appearing independently and with probability $p$. 

\begin{theorem}\label{thm:maingraph}
Let $A_n$ be the adjacency matrix of the random Erd\H{o}s--R\'enyi graph $G(n,p)$, and fix $\nu >0$.  There exist constants $C_{\ref{thm:main}}, c_{\ref{thm:main}}, c'_{\ref{thm:main}} > 0$, depending only on the subgaussian moment $B$, such that for
$$\frac{C_{\ref{thm:main}}\log^{7 + \nu} n}{n} \leq p \leq 1 - \frac{C_{\ref{thm:main}}\log^{7 + \nu} n}{n}$$
 and 
$$(np)^{-1/(7 + \nu)}  \leq \alpha \leq \frac{c'_{\ref{thm:main}}}{\log n},$$
the following holds for the gaps between the eigenvalues, $\delta_i= \lambda_{i+1} - \lambda_i$.  For any $\delta \geq \exp(-\alpha^{-1})$, 
\begin{equation}\label{eq:mainrandomgraph}
\sup_{1 \leq i \leq n-1} \P\left(\delta_i \leq \delta   \exp\left(-c_{\ref{thm:main}} \frac{\log(1/p)}{\log np}\right)  \sqrt{\frac{p}{n}} \right) \leq       \frac{\delta}{\alpha}
\end{equation}

\end{theorem}

\begin{remark}
Note that an upper bound on $p$ is necessary in this case as $p=1$ generates a \emph{deterministic} matrix with repeated eigenvalues.  Additionally, our argument can be easily applied to random perturbations of a finite rank matrix; see Remark~\ref{r:finite}.  However, for perturbations of an arbitrary matrix, new ideas are needed as many of the delicate net arguments cannot be adapted when the operator norm of the perturbed matrix is large. For dense random graphs, this was done in \cite[Theorem 2.6]{nguyen2017gaps}.
\end{remark}
\subsection{Non-degeneration of Eigenvectors and Nodal Domains of a Random Graph}\label{s:dll}

Consider the eigenfunctions of the Laplacian on a Riemannian manifold. The zero sets of these eigenfunctions partition the space into so-called \emph{nodal domains}. These domains are of great interest to geometers and have been intensively studied (see \cite{cheng1976eigenfunctions, muller1985number, lin1987second} and the references therein). Here we consider a discrete analogue, the nodal domains of eigenvectors for adjacency matrices of random graphs, which has its roots in graph theory and has recently found uses in data science \cite{helffer2009nodal, davies2000discrete, dekel2011nodaldomains}. Given an eigenvector $u$ of an adjacency matrix $A$, we call a subset $D$ of the vertices a \emph{weak nodal domain} if it is connected, $u(x) u(y) \ge 0$ for $x,y \in D$, and $D$ is a maximal subset under these two conditions. A \emph{strong nodal domain} is defined similarly using the strict inequality $u(x) u(y) > 0$. Dekel, Lee, and Linial  conjectured that the notions of strong and weak domains are equivalent for random graphs \cite{dekel2011nodaldomains}, and this was shown for $G(n,p)$ with constant $p$ in \cite{nguyen2017gaps}. A consequence of the following non-degeneration result is that we are able to resolve this conjecture for $p \ge C_{\ref{thm:main}} \log^{7 + \nu}(n)/n$.

\begin{theorem} \label{thm:nodaldomains} Let $A_n$ be the adjacency matrix of the random graph $G(n,p)$, and fix $\nu >0$. For any $D > 0$, there exists a $C=C(D)>0$ such that for $$\frac{C_{\ref{thm:main}}\log^{7 + \nu} n}{n} \leq p \leq 1 - \frac{C_{\ref{thm:main}} \log^{7 + \nu} n}{n},$$ the probability that there exists an eigenvector $v =(v_1, \dots, v_n)$ of $A_n$ with $|v_i| \le n^{-C}$ for some $i$ is at most $Cn^{-D}$.
\end{theorem} 

Theorem~\ref{thm:nodaldomains} provides a quantitative lower bound on the mass of the eigenvector components, complementing the vast literature on eigenvector delocalization, which provides upper bounds (see \cite[Section 4]{o2016eigenvectors} and \cite{benigni2020optimal}).

\begin{corollary}
For any $D > 0$, there exists $C=C(D)>0$ such that with probability at least $1 - Cn^{-D}$, the strong and weak nodal domains of $G(n,p)$ are the same. 
\end{corollary}

Arora and Bhaskara \cite{arora2011eigenvectors} showed that for random graphs $G(n,p)$ with $p \ge n^{-c}$, where $c$ is a constant that may be determined explicitly,\footnote{The authors give an exact value. However, the published version of an eigenvector delocalization estimate used to prove the result differs slightly from the version given in \cite{arora2011eigenvectors}, where it is cited by the authors in pre-publication form. The value of the constant should be adjusted in light of this.} all non-first eigenvectors of the adjacency matrix $A_n$ of $G(n,p)$ have exactly two weak nodal domains with high probability. Recall that since the adjacency matrix is not centered, the eigenvector corresponding to the largest eigenvalue behaves differently, tending to align itself with the all ones vector \cite{mitra2009entrywise}. Combining this result with our previous corollary yields the following simple statement.

\begin{corollary}
There exists $c>0$ such that the following holds. For any $D>0$ and $p \geq n^{-c}$, there exists $C= C(D)>0$ such that with probability at least $1 - C n^{-D}$, each eigenvector of $G(n,p)$ (except the first) has exactly two strong nodal domains which partition the vertices.
\end{corollary}

An identical non-degeneration result applies to matrices $M_n$ defined in Definition \ref{def:matrixmodel}.  
\begin{theorem}\label{thm:nondegeneration}
Fix $\nu >0$. For any $D > 0$, there exists a $C=C(D)>0$ such that for $$p \geq \frac{C_{\ref{thm:main}} \log^{7 + \nu} n}{n}$$ the probability that there exists an eigenvector $v =(v_1, \dots, v_n)$ of $M_n$ with $|v_i| \le n^{-C}$ for some $i$ is at most $Cn^{-D}$.
\end{theorem}

\begin{remark}
Theorems \ref{thm:nodaldomains} and \ref{thm:nondegeneration} represent specific examples of a range of possible results.  Specifically, varying $\alpha$ in Theorem \ref{thm:main} can lead to trade-offs in the size of the entries and the strength of the probability bound.  We have chosen to give a simple polynomial bound on the size and probability for the sake of simplifying the presentation.
\end{remark}

We also remark that nodal domains were studied in the recent work \cite{huang2019size}, which showed that there exists a constant $ c \ge 0$ such that for $ p \ge n^{-c}$ the two nodal domains identified in  \cite{arora2011eigenvectors} are balanced, meaning they each contain close to $n/2$ vertices with high probability. Further, \cite{rudelson2017delocalization} shows that, with high probability, any vertex is connected to some vertex in the other domain.

The remainder of the paper is organized as follows. In Section \ref{sec:proofstrategy}, we outline the key steps and intuition for the proof of Theorem \ref{thm:main}. In Sections \ref{sec:compressible} and \ref{sec:incompressible}, we prove several preliminary results about eigenvectors of sparse random matrices. In Section \ref{sec:mainproof}, we provide the proof of Theorem \ref{thm:main}.  In Section \ref{sec:randomgraph} we provide the necessary modifications to extend Theorem \ref{thm:main} to non-centered random matrices, such as the adjacency matrices of Erd\H{o}s--R\'enyi graphs, proving Theorem \ref{thm:maingraph}.  Finally, in Section \ref{sec:application}, we prove Theorem \ref{thm:nodaldomains}.

\section{Proof Strategy} \label{sec:proofstrategy}
The proof follows the same broad outline as \cite{luh2018sparse}.  
For $M_n$ as in Definition \ref{def:matrixmodel}, we decompose the matrix as
\begin{equation} \label{eq:decompositionMn}
M_n = 
\begin{pmatrix} 
  M_{n-1} & X \\ 
  X^T & m_{nn} 
\end{pmatrix},
\end{equation}
where $X = [ x_1, \dots, x_{n-1} ]  \in \mathbb{R}^{1 \times (n-1)}$.  For a matrix $W$, let ${\lambda_{n}(W) \ge \dots \ge \lambda_1(W)}$ be the eigenvalues of $W$. Fix an integer $i$ such that $1\le i\le n$ and let $v = (x , a)$ (where $x \in \mathbb{R}^{n-1}$ and $a \in \mathbb{R}$) be the unit eigenvector associated to $\lambda_i(M_n)$. By definition we have
$$
\begin{pmatrix}
M_{n-1} & X \\ 
X^T & m_{nn}
\end{pmatrix}
\begin{pmatrix}
x \\
a
\end{pmatrix} = 
\lambda_i(M_n) \begin{pmatrix}
x \\
a
\end{pmatrix}.
$$
For the top $n-1$ coordinates this gives (writing $\lambda_i(M_n)$ for $\lambda_i(M_n) \operatorname{Id}$)
$$
(M_{n-1} - \lambda_i(M_n)) x + a X = 0.
$$
Let $w$ be the eigenvector of $M_{n-1}$ corresponding to $\lambda_{i}(M_{n-1})$.  Multiplying on the left by $w^T$, we obtain
\begin{equation}\label{eq:eiggap1}
|a w^T X | = |w^T (M_{n-1} - \lambda_i(M_n)) x| = |\lambda_i(M_{n-1}) - \lambda_i(M_n)||w^T x|.
\end{equation}
By the Cauchy interlacing theorem, we have $\lambda_i(M_n) \leq \lambda_i (M_{n-1}) \leq \lambda_{i-1} (M_n)$. 

Since the entries of $M_n$ are subgaussian, we have with high probability that
$$
\lambda_i \in [-K \sqrt{pn}, K \sqrt{pn}]
$$
for some constant $K$ that depends only on the subgaussian moment $B$ of the entries.
Therefore, the average size of an eigenvalue gap is roughly $O\left(\frac{\sqrt{pn}}{n} \right) = O\left(\sqrt{\frac{p}{n}}\right).$  For any $\hat \delta >0$, let $\mathcal{E}_i = \mathcal{E}_i \left( \hat \delta \right)$ denote the event that 
$$
\lambda_{i+1} - \lambda_i  \leq \hat \delta  \sqrt{\frac{p}{n}}.
$$ We also let $\mathcal{G}_{i}$ be the intersection of the event $\mathcal{E}_i$ with the event that the eigenvector $v = (x, a)$ with eigenvalue $\lambda_i$ has $|a| \geq n^{-1/2}$.  Therefore, by (\ref{eq:eiggap1}) and using $|w^T x| \le1$, on the event $\mathcal{G}_{i}$, we have
\begin{equation}\label{e:seeme}
|w^T X| \leq \hat \delta \sqrt{p}.
\end{equation}
We wish to show this is unlikely. 

Recall that the theory of small ball probability (e.g.\ \cite{nguyen2013smallball}) examines the probability that a random variable takes values in a small interval.  
Therefore, we have reduced the problem to understanding the small ball probability of the inner product of a random vector with the eigenvector $w$.  It is known that this small ball probability is related to the amount of ``disorder" in the coordinates of the eigenvector.  Broadly speaking, a large amount of disorder implies the small ball probability is small. We deal with the case that $w$ has high disorder eigenvectors using these results. To exclude all eigenvectors with low disorder, we employ a covering argument, varying our approach according to the structure of the eigenvector.

The covering argument is completed in multiple stages.  For a fixed $\lambda$, we consider $M_n - \lambda \operatorname{Id}$ acting on the unit sphere, where $\operatorname{Id}$ is the identity operator.  Following the prescription initiated in a series of works \cite{litvak2005smallest, tao09littlewoodofford,rudelson2008littlewood, rudelson2008inverse,basak2017sparse,basak2018sharp}, we decompose the sphere into several sets that each offer their own advantages.  \emph{Compressible} vectors are those vectors that are close to \mbox{$m$-sparse vectors} for some parameter $m$.  
In \cite{basak2017sparse}, 
it was shown that the product of the matrix with a compressible vector has many large coordinates and therefore large $\ell_2$ norm.  We adapt this argument to our symmetric matrix case to exclude compressible vectors.
We next consider \emph{dominated} vectors, which are those vectors whose coordinates outside the $m$ largest coordinates have a small ratio of $\ell_2$ norm to $\ell_\infty$ norm.  This type of vector was introduced in \cite{basak2017sparse}.  As these vectors are also nearly sparse, they can be excluded similarly to the compressible vectors.  

Finally, for vectors that are neither compressible nor dominated, we use a stratification according to a measure of structure, the LCD. The LCD was introduced in \cite{rudelson2008littlewood} and is defined later.  As our random matrix is symmetric, there is dependence between the rows which prevents us from applying small ball probability estimates to each coordinate independently.\footnote{This obstacle is what prevents us from reaching the optimal threshold for $p$ by simply following the argument in  \cite{basak2017sparse}, which considered  non-symmetric matrices for $p \ge ( C \log n)/n$.} To address this problem, for a fixed $v$ we partition the coordinates of $v$ into small subsets; this is similar to the method used in \cite{vershynin2014symmetric}.  For a fixed subset, after conditioning on the columns of $M_n - \lambda$ outside of the subset, we can extract more independent coordinates to use in small ball estimates.  
There is some flexibility in the size of these subsets, and this ultimately results in the trade-off between the error probability and gap size in Theorem \ref{thm:main}.      

The previous steps are done for a fixed $\lambda$ and hold with exponentially high probability.  Taking a union bound over a fine enough net of the interval $[-K \sqrt{pn}, K \sqrt{pn}]$ completes the argument.

A similar approach was applied in \cite{luh2018sparse}, under the assumption that $p \geq n^{-1+ \eps}$ for some $\eps >0$ and therefore small polynomial terms could often be neglected.  In our current setting, where $p$ is on the order of $\log^C n/n$, it turns out that the above decomposition is insufficient primarily because the vectors that are not dominated or compressible can have a wide range of $\ell_2$ mass in their coordinates outside of the $m$ largest.  
Therefore, we further decompose the vectors by their $\ell_2$ mass in the relevant coordinates.  Working in each of these classes allows some key technical estimates that bypass the small polynomial losses from \cite{luh2018sparse}.  These technical improvements generate the improvement in the range of sparsity and the error probability.  Furthermore, in \cite{luh2018sparse}, the result was only concerned with a non-zero separation of the eigenvalues.  A more careful accounting of the small ball probability greatly improves the (implicit) small ball estimate in \cite{luh2018sparse}.

\section{Compressible and Dominated Vectors} \label{sec:compressible}

The goal of this section is to prove Proposition~\ref{prop:eigvecnotcomp}, which shows that any eigenvector of $M_n$ cannot be close to a sparse vector, in a certain quantitative sense (with high probability). Before proceeding to its proof, we introduce a few necessary definitions and lemmas.

\subsection{Decomposition of the sphere} \label{sec:decomp}
We now formally define the decomposition of the unit sphere used in the proof sketch of Section~\ref{sec:proofstrategy}.
\begin{definition}\label{d:sparse}
Fix $m< n$. The set of $m$-sparse vectors is given by 
$$
\text{Sparse}(m) = \{x \in \R^n: |\supp(x)| \leq m\}.
$$
Furthermore, for $\delta > 0$, we define the compressible and incompressible vectors by
$$
\Comp(m, \delta) = \{x \in \mathbb{S}^{n-1}: \exists y \in \text{Sparse}(m) \text{ such that } \|x - y\|_2 \leq \delta\},
$$
and 
$$
\Incomp(m, \delta) = \{x \in \mathbb{S}^{n-1}:  x \notin \Comp(m,\delta) \}.
$$
\end{definition}
For any $ 1 \le n \le n'$,  we let $[n]$ denote the set $\{1, 2, \dots , n\}$ and $[n:n']$ denote the set $\{n, n+1, \dots , n'\}$.

\begin{definition}\label{d:vectorsubscript}
For any $x \in \mathbb{S}^{n-1}$, let $\pi_x : [n] \rightarrow [n]$ be a permutation which arranges the absolute values of the coordinates of $x$ in  non-increasing order. 
For $1 \leq m \leq m' \leq n$ denote by $x_{[m:m']} \in \R^n$ the vector with coordinates 
$$
x_{[m:m']}(j) = x_j \cdot \mathbbm{1}_{[m:m']}(\pi_x(j)).
$$
For any $c < 1$ and $m \leq n$, define the set of vectors with dominated tail by
$$
\Dom(m, c) = \{x \in \S^{n-1}: \|x_{[m+1:n]} \|_2 \leq c \sqrt{m} \|x_{[m+1:n]} \|_\infty \}.
$$
\end{definition}
This definition was first given in \cite{basak2017sparse}. Like compressible vectors, vectors with dominated tail are close to being sparse, though in a different  way. This approximate sparsity facilitates the proof of the following key bound, Proposition~\ref{prop:compressible}.

\subsection{Bounds for compressible and dominated vectors}

We first state a high probability bound on the operator norm of $M_n$, which was defined in Definition~\ref{def:matrixmodel}.

\begin{lemma}[{\cite[Proposition 5.2]{luh2018sparse} \text{ and } \cite[Proposition 1.10]{wei2017investigate}}]\label{l:opnorm}
For $M_n$ defined in Definition \ref{def:matrixmodel}, there exist constants $C_{\ref{l:opnorm}}, K, c_{\ref{l:opnorm}} >0,$ depending only on the subgaussian moment $B$, such that for 
$
p \geq \frac{C_{\ref{l:opnorm}} \log n}{n}
$ 
and
$n \ge (c_{\ref{l:opnorm}})^{-1}$,
$$
\P(\|M_n \| \geq K \sqrt{pn} ) \leq \exp(-c_{\ref{l:opnorm}} p n).
$$
\end{lemma}
\noindent For the remainder of this work, all references to the constant $K$ refer to the $K$ provided by Lemma~\ref{l:opnorm}.

The compressible and dominated vectors were previously resolved in \cite{luh2018sparse} down to the optimal scale $p \geq C \log n/ n$. Given some $\bar{C}_{\ref{prop:compressible}}>0$, we define the parameters 

 $$ \ell_0 = \left \lceil \frac{\log 1/(8p)}{\log \sqrt{pn}} \right \rceil, \qquad  \rho = (\bar{C}_{\ref{prop:compressible}})^{-\ell_0 - 6}.$$

\begin{proposition}[{\cite[Proposition 5.3]{luh2018sparse}}]\label{prop:compressible}
There exist constants $C_{\ref{prop:compressible}}, \bar{C}_{\ref{prop:compressible}}, c_{\ref{prop:compressible}}, c'_{\ref{prop:compressible}},  > 0$, depending only on the subgaussian moment $B$ of Definition~\ref{def:matrixmodel}, such that the following holds. If $p, m, \lambda$ satisfy
\begin{equation}\label{e:m}
p \geq \frac{C_{\ref{prop:compressible}} \log n}{n}, \quad p^{-1} \leq m \leq c_{\ref{prop:compressible}}n  , \text{ and } \lambda \in [-K \sqrt{pn}, K \sqrt{pn}],
\end{equation}
then with probability at least $1 - \exp(- {c}'_{\ref{prop:compressible}} pn)$,
$$  \|(M_n - \lambda) x\|_2 \ge  c_{\ref{prop:compressible}} \rho \sqrt{pn}$$
for all $x \in  \Comp(m, \rho) \cup \Dom(m, c'_{\ref{prop:compressible}})$ and $n > (c'_{\ref{prop:compressible}})^{-1}$.

\end{proposition}

\begin{remark}\label{r:rhoremark}
Note that if $p \geq n^{-1+c}$ for some constant $c>0$, then $\rho$ is bounded below by a constant.  At the optimal scale $p = C \log n/n$, there exist constants $C_1, C_2 . c_1, c_2 >0$ such that
$$
C_1 \exp(- c_1 \log n/ \log \log n) \leq \rho \le  C_2 \exp(-c_2 \log n/ \log \log n).
$$
\end{remark}

We now come to the main result of this section, which combines the previous two proposition to exclude the possibility of compressible or dominated eigenvectors.

\begin{proposition}\label{prop:eigvecnotcomp}
Let be $M_n$ as in Definition~\ref{def:matrixmodel} with $p \geq C_{\ref{prop:compressible}} \frac{\log n}{n}$. For 
$p^{-1} \leq m \leq c_{\ref{prop:compressible}}n$
and $n \ge (c_{\ref{prop:eigvecnotcomp}})^{-1}$,
$$
\P(\emph{there exists an eigenvector } v \in \Comp(m, \rho) \cup \Dom(m, c_{\ref{prop:compressible}}) ) \leq \exp(- c_{\ref{prop:eigvecnotcomp}} p n)
$$
for some constant $c_{\ref{prop:eigvecnotcomp}} > 0$.
\end{proposition}
\begin{proof}
Let $\mathcal{N}$ denote a $c_{\ref{prop:compressible}} \rho \sqrt{pn}$-net of the interval $[-K \sqrt{pn}, K \sqrt{pn}]$ with 
\begin{equation}
| \mathcal{N}|  \le \frac{4 K}{c_{\ref{prop:compressible}} \rho}.
\end{equation}

 If there exists a compressible or dominated eigenvector $v$ with eigenvalue $\lambda \in [-K \sqrt{pn}, K \sqrt{pn}]$, then there exists a $\lambda_0 \in \mathcal{N}$ such that
$$
\|(M_n - \lambda_0) v\|_2 = \|(\lambda - \lambda_0) v\|_2 \leq c_{\ref{prop:compressible}} \rho \sqrt{pn}.
$$
By a union bound and Proposition \ref{prop:compressible}, the probability of this event is bounded by 
$$
|\mathcal{N}| \exp(-c_{\ref{prop:compressible}} p n) \leq \exp(- c_{\ref{prop:eigvecnotcomp}} p n).
$$
for large enough $C_{\ref{prop:compressible}}$ and small enough $c_{\ref{prop:eigvecnotcomp}}$;
to bound $|\mathcal{N}|$, we used Remark~\ref{r:rhoremark}.
Finally, the event that that there exists an eigenvalue outside of the interval $[-K \sqrt{pn}, K \sqrt{pn}]$ is bounded by $\exp(-c_{\ref{l:opnorm}} pn)$, by Lemma~\ref{l:opnorm}.  Shrinking $c_{\ref{prop:eigvecnotcomp}}$ allows us to take a union bound to include this event, and concludes the proof.
\end{proof}

\section{Incompressible Vectors} \label{sec:incompressible}
In this section, we show that $M_n$ does not have structured eigenvectors. We begin with Section~\ref{s:sbp}, where we elucidate the connection between small ball probability and our measure of structure, the Least Common Denominator (LCD). Section~\ref{s:smallLCDdecomp} and Section~\ref{s:smallLCD} are devoted to the proof of Proposition~\ref{p:smallLCDallLevels}, which shows it is unlikely an eigenvector of $M_n$ has an LCD lying in a given level set. This proposition is the main technical achievement of this section. Finally, we derive Proposition~\ref{p:eigvectors} as a straightforward consequence of Proposition~\ref{p:smallLCDallLevels} and a union bound, which excludes the possibility of structured eigenvectors altogether. Together with Proposition~\ref{prop:eigvecnotcomp}, Proposition~\ref{p:eigvectors} will allow us to complete the outline of Section~\ref{sec:proofstrategy} and prove our main theorems in the next section.

\subsection{Small Ball Probability}\label{s:sbp}
Recall from the proof sketch in Section \ref{sec:proofstrategy} that we wish to bound the probability that the inner product of an eigenvector and a random vector is small.  This motivates the definition of L\'evy concentration, which bounds the small ball probabilities of a random vector $Z$.
 
\begin{definition}
	The \emph{L\'evy concentration} of a random vector $Z \in \mathbb{R}^n$ is defined to be
	$$
	\mathcal{L}(Z, \eps) = \sup_{u \in \mathbb{R}^n} \P(\|Z- u\|_2 \leq \eps).
	$$
\end{definition}
When $X$ is a random vector and $v$ is a fixed vector, the structure of $v$ will greatly influence the L\'evy concentration of the random variable $v \cdot X$. To formalize this concept, we begin with a measure of arithmetic structure for a  unit vector.  

\begin{definition}[{\cite[Definition 6.1]{vershynin2014symmetric}}]\label{def:LCD}
Let $p$ be as in Theorem~\ref{thm:main}. We define the least common denominator (LCD) of $x \in \S^{n-1}$ as
$$
D(x) = \inf\left\{\theta > 0 : \operatorname{dist}(\theta x, \Z^n) < \left( {\frac{\log_{+} (\sqrt{\gamma p}\theta)}{\gamma p}  } \right)^{1/2}\right\},
$$
where $\gamma$ is an appropriate constant that is defined in Remark \ref{r:gamma} below.
\end{definition}

\begin{remark}\label{r:gamma}
	There exist constants $\gamma, \bar{\eps}_0 \in (0,1)$ such that for any ${\eps \leq \bar{\eps}_0}$, $$\mathcal{L}(\xi \chi, \eps) \leq 1-\gamma p,$$ where $\chi$ is a Bernoulli random variable such that $\P(\chi = 1) = p$ and $\xi$ is a subgaussian random variable with unit variance.  We fix such a $\gamma$ in Definition \ref{def:LCD}.
\end{remark}

\begin{proposition}[{\cite[Proposition 4.2]{basak2017sparse}}] \label{prop:smallballprobability}
Let $X \in \mathbb{R}^n$ be a random vector with i.i.d.\ coordinates of the form $\xi_j \chi_j$, where the $\chi_j$'s are Bernoulli random variables with $\P(\chi_j = 1) = p$ and the $\xi_j$'s are random variables with unit variance and finite fourth moment.  Then for any $v \in \S^{n-1}$, 
	$$
	\mathcal{L} \left( X \cdot v, \sqrt{p} \eps \right) \leq C_{\ref{prop:smallballprobability}} \left( \eps + \frac{1}{\sqrt{p} D(v)} \right),
	$$
	where $C_{\ref{prop:smallballprobability}}$ depends only on the fourth moment of $\xi$.
\end{proposition}

We may tensorize Proposition \ref{prop:smallballprobability} to obtain a bound on the L\'evy concentration of $M_n x$. The argument is almost identical to the proof of \cite[Proposition 4.3]{basak2017sparse}, and we note only the necessary modifications here. Recall the notation $x_{[m:m']}$ from Definition~\ref{d:vectorsubscript}. For any index set $J \subset [n]$, we extend this notation to $x_J$ in the canonical way.

\begin{proposition}[Small ball probabilities of $M_n x$ via regularized LCD] \label{prop:smallballprob}
There exists a constant $C_{\ref{prop:smallballprob}}$ such that for any $\alpha, \eps > 0$ and index set $I$ of size $\lceil \alpha n \rceil$, 
$$
  \LL(M_n x, \eps \|v_I\|_2 \sqrt{p n}) \leq C_{\ref{prop:smallballprob}}^{n- \lceil \alpha n \rceil} \left( \eps + \frac{1}{\sqrt p D(v_I/\|v_I\|_2)} \right)^{n- \lceil \alpha n \rceil}.
$$
\end{proposition}
\begin{proof}[Proof Sketch] We first observe that conditioning on elements of $M_n$ never decreases (and may increase)  $\LL(M_n x, \eps \|v_I\|_2 \sqrt{p n})$. We therefore condition on all elements not in columns indexed by elements of $I$, and also condition the elements whose indices $(i,j)$ satisfy $i,j \in I$. The remaining elements are i.i.d.\ and consist of $n - \lceil \alpha n \rceil$ rows. The remainder of the argument is nearly identical to the one leading to \cite[Proposition 4.3]{basak2017sparse}, where an analogous statement was shown for non-symmetric matrices.
\end{proof}

The following lemma provides a lower bound for the LCD in terms of the $\ell^\infty$ norm.
\begin{proposition}[Lemma 6.2, \cite{vershynin2014symmetric}] \label{prop:LCDlowerbound}
For all $x \in \S^{n-1}$,
$$
D(x) \geq \frac{1}{2 \|x\|_\infty}.
$$
\end{proposition}

As in \cite{vershynin2014symmetric}, we define a regularized version of the LCD. However, our definition is slightly different than the one in \cite{vershynin2014symmetric}. Recall the notation $\Incomp(m ,\delta)$ given after Definition~\ref{d:sparse}, and observe that the set $I_0$ in the following definition takes a distinguished role and is not included in the maximum. Here, $k_0$ represents a parameter that will be fixed later, in the material preceding \eqref{e:kbounds}.

\begin{definition}[Regularized LCD]					\label{def reg LCD}
  Let $\{I_j\}_{j=0}^{k_0}$ be any partition of $[n]$ with $k_0$ elements..
  We define the {\em regularized LCD} of a vector $v \in \Incomp(m ,\delta)$ as
  $$
  \Dhat(v) =  \Dhat(I, v)  = \max_{1 \leq j \leq k_0}  D \big(x_{I_j}/\|x_{I_j}\|_2\big).
  $$\end{definition}

In our use of Definition~\ref{def reg LCD} below, $I_0$ will be (approximately) the $m$ largest coordinates of $v$. Hence $\Dhat(v)$ gives a measure of the structure of the elements of $v$ left over after approximating $v$ by an $m$-sparse vector.  

\subsection{Decomposition of Incompressible Vectors}\label{s:smallLCDdecomp}
In this section, we define a way to decompose incompressible vectors, which is used in the proof of Proposition~\ref{prop:smallLCD} below. In order to give this decomposition, we first introduction a classification of the incompressible vectors, which allows us to control the amount of mass that is not in the $m$ largest coordinates.  
\begin{definition}
For $\rho \leq \rho_1 \leq \rho_2 \le 1$ and $c <1$, define
\begin{multline*}
\Incomp_{\rho, c}(m, \rho_1, \rho_2) = \\ \left \{v \in \mathbb{S}^{n-1} \cap \left(\Comp(m, \rho) \cup \Dom(m, c) \right)^c \colon \rho_1 \leq \|v_{[m+1:n]}\|_2 < \rho_2 \right \}.
\end{multline*}
\end{definition}  
\begin{remark}
By definition, $\| v \|_2 \le \rho $ for any $v \in \Comp(m, \rho)$, which gives rise to the condition $\rho \le \rho_1$ in the preceding definition.
\end{remark}
We will consider the sets of incompressible vectors $\Incomp_{\rho , c'_{\ref{prop:compressible}}}(m, 2^{j-1} \rho, 2^j \rho)$ for $j \in \mathbb N$, where $m$ is a parameter that will be chosen later. For brevity, we introduce the shorthand
$$ \IncompTwo =\Incomp_{\rho , c'_{\ref{prop:compressible}}}(m, 2^{j-1} \rho, 2^j \rho).$$
For the remainder of this section we primarily use the fact that the vectors in $\Incomp(m, 2^{j-1} \rho, 2^j \rho)$ are not dominated. That they are not compressible is used only in the proof of Proposition~\ref{p:smallLCDallLevels}.

We begin with a straightforward upper bound. Recall $\rho$ was defined in Proposition~\ref{prop:compressible}. Fix $j \in \mathbb Z$ and consider a vector ${v \in \Incomp(m, 2^{j-1} \rho, 2^j \rho)}$.  Since $v \notin \Dom(m , c'_{\ref{prop:compressible}})$, 
$$
\|v_{[m+1:n]} \|_2 >  c'_{\ref{prop:compressible}} \sqrt{m} \|v_{[m+1:n]}\|_\infty. 
$$
Furthermore, since $\|v_{[m+1:n]}\|_2 < 2^{j} \rho$ by definition,
\begin{equation}\label{e:infinitybound}
  c'^{-1}_{\ref{prop:compressible}} \frac{2^j \rho}{\sqrt{m}}  > \|v_{[m+1:n]}\|_\infty.
\end{equation}
On the other hand, we can also find a large set of coordinates that are uniformly lower-bounded.
\begin{lemma} \label{l:largeset}
For $v \in \IncompTwo$, the set $$
\sigma(v) = \left \{i \in [n]: |v_i| \geq  \frac{2^{j-1} \rho}{2 \sqrt{n}} \text{ and } i \in  \pi^{-1}_v\left([m+1:n]\right) \right\}
$$ 
satisfies $|\sigma(v)| \geq (c'_{\ref{prop:compressible}} )^2 m/ 8 $.
\end{lemma}
\begin{proof}
For the sake of contradiction, assume that $|\sigma(v)| < (c'_{\ref{prop:compressible}})^2 m/8$.  Then by \eqref{e:infinitybound},
$$
\|v_{[m+1:n]}\|_2 \le  \sqrt{\|v_{[m+1:n]}\|_\infty^2 |\sigma(v)| + n \frac{2^{2(j-1)} \rho^2}{4n} }  < 2^{j-1} \rho,
$$ 
contradicting the definition of $\IncompTwo$.
\end{proof}

We now define a partitioning procedure.  For this, we introduce some new notation.
\begin{definition}
For a set $I \in [n]$ with $|I| \geq k_2 > k_1 $, we use $I_{\langle k_1: k_2\rangle}$ to denote all the elements from the $k_1$-th to the $k_2$-th in $I$ (inclusive), where we order the elements from least to greatest.  For example, if $I = \{2, 4, 5, 6, 9\}$ then $I_{\langle 2:4 \rangle} = \{4, 5, 6\}$.  
\end{definition}

 Let $v\in \mathbb{S}^{n-1}$ be a vector, let $\omega = \omega(n) $ be a parameter satisfying
 $$ n^{-1/7} \le \omega\le \frac{1}{\log n},$$
and set $m= \omega n$. We define $k_0$ as the largest number of disjoint subsets with $\lceil \omega n \rceil$ elements one can have of $[n]$ whose union does not contain the indices of the $m$ largest elements of $v$.
We  consider disjoint index sets $I_1, \dots, I_{k_0}$, each of size $\lceil \omega n \rceil$, each not containing any indices of the $m$ largest elements of $v$. Therefore,
\begin{equation}\label{e:kbounds}
\frac{1}{2\omega} \leq \left\lfloor \frac{n-m }{\lceil \omega n \rceil} \right\rfloor = k_0 \le \frac{1}{\omega}.
\end{equation}

In our definition, the index sets $I_j$ depend on $v$, but we suppress this dependence in the notation.  For a vector $v\in \mathbb{S}^{n-1}$, let $\tau(v)$ denote the set of indices of the $m$ largest coordinates.  By Lemma \ref{l:largeset}, we can choose a subset $\widehat \sigma (v) \subset \sigma(v)$ of size exactly $\lceil (c'_{\ref{prop:compressible}} )^2 m/8 \rceil$, where $\sigma(v)$ was defined in the statement of that lemma. We observe that $\widehat \sigma (v)$ and $\tau(v)$ are disjoint.

Let $\overline{\sigma}(v) = [n] \setminus (\tau(v) \cup \widehat \sigma(v)).$ For $1 \leq k < k_0$, we define
$$r' =\left \lceil \frac{(c'_{\ref{prop:compressible}} )^2 m}{8 } \right \rceil ,\quad r = \left\lfloor  \frac{r'}{k_0} \right\rfloor ,\quad s = \lceil \omega n \rceil -r,$$
\begin{equation}\label{e:partition}
I_k = \widehat \sigma(v)_{\left \langle 1 + (k-1) r :  k  r \right \rangle} \cup \overline{\sigma}(v)_{\left \langle  1 + (k-1) s:  k s \right \rangle}.
\end{equation}
For the rest of this work, we drop floor and ceiling functions because they do not influence the argument in a substantial way.

Finally, we define $I_0 = [n] \setminus \cup_{k=1}^{k_0} I_k$.  In words, $I_0$ contains the $m$ largest coordinates and the smaller coordinates left over from divisibility issues.  In particular, $|I_0| \leq m + \lceil \omega n \rceil$.  Since the sets $I_k$ were chosen to be disjoint for $k\ge 1$, it follows that $\{ I_k\}_{k=0}^{k_0}$ is a partition of $[n]$. 

The primary objective of this partition is recorded in the following lemma, where we also define the constants $\rho'_j$. 

\begin{lemma}\label{ikbounds}
For $v \in \IncompTwo$ and $1 \leq k \leq k_0$,
\begin{equation} \label{eq:l2bound}
\rho'_j :=  c'_{\ref{prop:compressible}} 2^{j-3} \rho \omega \leq \|v_{I_k}\|_2 \leq c'^{-1}_{\ref{prop:compressible}} 2^{j} \rho 
= 2^3 (c'_{\ref{prop:compressible}})^{-2} \rho'_j  \omega^{-1}.
\end{equation}
Also,
$$
\widehat{D}(v ) \geq (c'_{\ref{prop:compressible}} )^2  2^{-5} n^{1/2} \omega^{3/2}.
$$
\end{lemma}

\begin{proof}
The bounds on $\|v_{I_k}\|_2$ follow from the coordinate-wise bounds of our construction. For the lower bound, we ignore all elements not in $\widehat \sigma(v)$. We obtain
$$
\sqrt{\frac{|\widehat \sigma(v)|}{k_0}}  \frac{2^{j-1} \rho}{2 \sqrt{n}}  \leq \|v_{I_k}\|_2 \leq \sqrt{|I_k|} \|v_{[m+1:n]}\|_\infty.
$$
The claim \eqref{eq:l2bound} then follows from Lemma~\ref{l:largeset}, \eqref{e:infinitybound}, and \eqref{e:kbounds}.

For the second claim, applying Proposition~\ref{prop:LCDlowerbound} and recalling Definition~\ref{def reg LCD} yields
\begin{equation}\label{e:Dlower}
\widehat{D}(v ) \geq \min_{k \ge 1} \left\{ \frac{\|v_{I_k}\|_2}{2 \|v_{I_k}\|_\infty}\right\}.
\end{equation}
Then the claim follows from the lower bound on $\|v_{I_k}\|_2$ in the previous paragraph and \eqref{e:infinitybound}.
\end{proof}

\subsection{Vectors with Small LCD}\label{s:smallLCD}
We now exclude vectors with small regularized LCD as potential eigenvectors of $M_n$. 
This is the content of the next proposition, Proposition~\ref{prop:smallLCD}, which shows that any vector in $\IncompTwo$ with small regularized LCD is unlikely to be near an eigenvector. We first define level sets of vectors according to their regularized LCD.
\begin{definition}
For any $L> 0$, we define the level sets
$$\label{e:SL}
S_L = \{v \in \Incomp(m, \rho): L \leq \widehat{D}(v) < 2L\}.
$$
\end{definition}

We also require a preliminary lemma. 
Recall $\gamma$ was defined in Remark~\ref{r:gamma}.

\begin{lemma}[Lemma 6.13, \cite{luh2018sparse}]\label{thenets}
Let $\omega >0$, and let $f(n)$ be a function such that $${\lim_{n \rightarrow \infty} f(n) = \infty}.$$ Then for $L > f(n)$, the set of unit vectors
$$
\{v \in \S^{\omega n -1}: f(n) \leq D(v) \leq L\}
$$
admits a $\beta$-net of size at most
$$
\left(12 + \frac{\bar{c} L}{\sqrt{\omega n}} \right)^{\omega n} \log(L)
$$
where $\bar{c}>0$ is a universal constant and 
$$
\beta = \frac{2 \sqrt{\log(2 \sqrt{\gamma p} L)}}{L \sqrt{\gamma p}}.
$$
\end{lemma}

We now state and prove the main technical result of this section. Recall $S_L$ was defined in \eqref{e:SL}, $\rho'_j$ was defined in Lemma~\ref{ikbounds}, and $K$ is the constant given by Lemma~\ref{l:opnorm}.
\begin{proposition} \label{prop:smallLCD}
Fix $\nu>0$. There exist constants $C_{\ref{prop:smallLCD}}, c_{\ref{prop:smallLCD}}, c'_{\ref{prop:smallLCD}}, \tilde{c}_{\ref{prop:smallLCD}} >0$ such that for $p \geq C_{\ref{prop:smallLCD}} \frac{\log^{7 + \nu} n}{n}$, $\lambda \in [-K \sqrt{pn}, K \sqrt{pn}]$, $j \in \N$, and for any
$$
(np)^{-1/(7 + \nu)} \leq \alpha \leq \frac{c'_{\ref{prop:smallLCD}}}{\log n}
$$
and
$$
\tilde{c}_{\ref{prop:smallLCD}} \alpha^{3/2} n^{1/2} \leq L \leq p^{-1/2} \exp(\alpha^{-1}),
$$ 
the following holds for $n \ge (c'_{\ref{prop:smallLCD}})^{-1}$:
$$
\P\Big(\exists v \in \Incompalphan \cap S_L \text{ s.t. } \|(M_n -\lambda)v \|_2 \leq c_{\ref{prop:smallLCD}}\eps_0 \rho'_j \sqrt{pn}  \Big) \leq \exp(-c'_{\ref{prop:smallLCD}}n),
$$
where 
$$
\eps_0 (L) = \min\left\{\frac{c'_{\ref{prop:smallLCD}} \sqrt{\alpha n}}{L}, \frac{c'_{\ref{prop:smallLCD}} \sqrt{\log r} (\log \log n)}{ \alpha^{2} r}\right\}
\text{ and } 
r= \frac{ \tilde{c}_{\ref{prop:smallLCD}}}{2}  \alpha^{3/2} (np)^{1/2}.
$$
\end{proposition}  

\begin{proof}

We set $m= \alpha n$, and define
$$\mathcal K = \IncompTwo \cap S_L.$$  
In outline, this proof implements the following steps:
\begin{enumerate}
\item Construct a suitable net $\mathcal M$ for $\mathcal K$.
\item Upper bound the size of $\mathcal M$.
\item Show the claim holds for all $v \in \mathcal M$. 
\item Extend the result from all $v\in \mathcal M$ to all  $v\in\mathcal K$.
\end{enumerate}

For Step 1, let $v\in \mathcal K$ be a vector and consider the partition $\{ I_k \}_{k=0}^{k_0}$ of the coordinates of $v$ constructed in \eqref{e:partition} with the parameter $\omega = \alpha$.  For the coordinates $I_0$, by a standard volume estimate,\footnote{See for example \cite[(5.7)]{pisier1999volume}.} there exists a $c'_{\ref{prop:smallLCD}}\rho'_j \eps_0/10K$-net, $\mathcal{N}_0$, of the values $[0,1]$ such that
$$
|\mathcal{N}_0| \leq \left(\frac{30 K}{c'_{\ref{prop:smallLCD}}\eps_0 \rho'_j} \right)^{m + \alpha n},
$$
where we recall $|I_0| \le m + \alpha n$.

For the coordinates in $I_k$ with $k \ge 1$, we use a construction that exploits the LCD structure. Observe that the hypothesis of Lemma~\ref{thenets} holds for $v_{I_k}/\| v_{I_k} \|_2$ because 
\begin{equation}\label{e:Llower}
{D(v_{I_k}/\| v_{I_k} \|_2) \ge (c'_{\ref{prop:compressible}} )^2 \frac{m}{2^5}\sqrt{\frac{\alpha }{n}} } =\frac{c'_{\ref{prop:compressible}}}{2^5} \alpha^{3/2} n^{1/2} ,\end{equation}
 as shown in the proof of Lemma~\ref{ikbounds} (see \eqref{e:Dlower}), and the lower bound tends to infinity as $n \rightarrow \infty$. For $I_k$ with $k\ge 1$, let $\mathcal{N}_k$ denote the $\beta$-net guaranteed by Lemma~\ref{thenets} applied to $v_{I_k}/\| v_{I_k} \|_2$.\footnote{Observe we are applying this lemma when the upper limit is $2L$, according to the definition of $S_L$, not $L$. The definition of $\beta$ is adjusted accordingly below.}

We next implement a net of scaling factors.  Let $\mathcal{J}$ be a $c'_{\ref{prop:smallLCD}}\eps_0 \rho'_j/ 10 K k_0$-net of $[0, 1]$ such that
$$
|\mathcal{J}| \leq \frac{30 K k_0}{c'_{\ref{prop:smallLCD}} \eps_0 \rho'_j}.
$$
As observed earlier, the partition $\{I_k\}_{k\ge 0}$ of the coordinates of $v$ is entirely determined by the sets of indices $\tau$ and $\sigma$.  To approximate all $v \in \mathcal K$, we define the preliminary set
\begin{equation}\label{e:choosetau}
\mathcal{M}' = \bigcup_{\tau, \sigma \in [n]: |\tau| = m, |\sigma|= m/4} \left\{x_0 + \sum_{k=1}^{k_0} t_k y_k: x_0  \in \mathcal{N}_0 , y_k \in \mathcal{N}_{k}, t_k \in \mathcal{J} \right\}.
\end{equation}
We currently have no guarantee that $${\mathcal{M}' \subset  \IncompTwo \cap S_L}.$$ 
However, this is easily fixed. 
If there exists $x\in S_L$ such that $$\| x - m \|_2 \le \frac{c_{\ref{prop:smallLCD}} \rho'_j \eps_0}{15K},$$ we replace $m$ by any such $x$. Otherwise, we discard $m$. This creates a new net $\mathcal M$ such that $|\mathcal M | \le |\mathcal M '|$. This completes Step 1.

We now enter Step 2 of the proof and upper bound the size of $\mathcal M$.  
We may combinatorially determine the size of $\mathcal M$ using the sizes of the $\mathcal N_k$ and $\mathcal J$. This leads to the following bound on the cardinality of our net:
\begin{align}\label{e:theprod}
|\mathcal{M}| &\leq \binom{n}{m} \binom{n}{m/4} \left(\frac{30 K}{c'_{\ref{prop:smallLCD}}\eps_0 \rho'} \right)^{m + \alpha n} \prod_{k=1}^{k_0} \left[ \left(12 + \frac{ \bar{c} 2L}{\sqrt{\alpha n}} \right)^{\alpha n} \log(L) \frac{30 K k_0}{c'_{\ref{prop:smallLCD}} \eps_0 \rho'_j} \right].
\end{align}
The combinatorial factors come from the choices of $\tau$ and $\sigma$ in \eqref{e:choosetau}.

We now proceed to simplify this bound. From the elementary bound $${\binom{n}{k} \le \exp(k\log(en/k))}$$ we have the following exponential bound for $|\mathcal M|$:
\begin{multline*}
|\mathcal{M}| \leq \exp \Big( 2m \log(4en/m)  + (m + \alpha n + k_0) \log(30K/c'_{\ref{prop:smallLCD}}\eps_0 \rho'_j) 
\\ + k_0 \log(\log(L)) + k_0 \log(k_0) \Big) \times \left(12 + \frac{\bar{c}2 L}{\sqrt{\alpha n}} \right)^{n-m}.
\end{multline*}
For the second factor, we recalled that $| I_0 | \ge m$, so that the product from $1$ to $k_0$ in \eqref{e:theprod} has at most $n-m$ individual terms.
 Using ${L \le \exp (2 \alpha^{-1})}$, ${m=\alpha n}$, $k_0 \le \alpha n$, and $k_0 \le \alpha^{-1}$ (from \eqref{e:kbounds}), we find
\begin{multline*}
|\mathcal{M}| \leq  \exp \left( n \left[ 2 \alpha  \log(4e/\alpha ) + 3 \alpha \log(30K/c'_{\ref{prop:smallLCD}}\eps_0\rho'_j)  
+ \frac{1}{n\alpha} \log(2 /\alpha^2) \right] \right)\\ \times  \left(12 + \frac{\bar{c} 2 L}{\sqrt{\alpha n}} \right)^{n-m}.
\end{multline*}

\noindent Recall that $\rho'_j$ was defined in terms of $\rho$ in Lemma~\ref{ikbounds}, and $\log(1/\rho) = O(\log n/ \log \log n)$ by Remark~\ref{r:rhoremark}. Note also that $\log(1/\alpha) = O(\log n)$. Then there exists $C>0$ such that
$$
2 \alpha  \log(4e/\alpha) \le C, \qquad 3 \alpha \log(30K/c'_{\ref{prop:smallLCD}}\rho'_j) \le C, \qquad \frac{1}{n\alpha} \log(2 /\alpha^2) \le C.
$$
From this, we find
\begin{equation}\label{e:mbound11}
|\mathcal{M}| \leq  \exp \left( n \left[ C  + 3\alpha  \log(1/\eps_0)  \right] \right) \times  \left(12 + \frac{\bar{c} 2 L}{\sqrt{\alpha n}} \right)^{n-\alpha n}.
\end{equation}
This completes Step 2.

We now begin Step 3 of the outline and prove the result for all the points in our net $\mathcal M$. Set $$
P = \P\left(\exists x \in \mathcal{M} \text{ s.t. } \|(M_n - \lambda)x\|_2 \leq {c_{\ref{prop:smallLCD}}\eps_0 \rho'_j \sqrt{pn}}\right).
$$
By Proposition \ref{prop:smallballprob} applied with $\eps = c_{\ref{prop:smallLCD}}\eps_0$, for any $v \in \mathcal M$ and $k$ such that $1\le k \le k_0$,
$$
\P\left( \| (M_n - \lambda) v \|_2 \le  c_{\ref{prop:smallLCD}}\eps_0 \rho'_j \sqrt{pn} \right) \le C_{\ref{prop:smallballprob}}^{n- \lceil \alpha n \rceil} \left( c_{\ref{prop:smallLCD}}\eps_0 + \frac{1}{\sqrt p D(v_{I_k}/\|v_{I_k}\|_2)} \right)^{n- \lceil \alpha n \rceil},
$$
where we recall from Lemma~\ref{ikbounds} that $\rho'_j \le \| v_{I_k} \|_2$. Since $v\in S_L$, by the definition of $S_L$ we find there exists $1\le k \le k_0 $ such that $ D(v_{I_k}/\|v_{I_k}\|_2) > L$. We use this $k$ in the above expression to find 
$$
\P\left( \| (M_n - \lambda) v \|_2 \le  c_{\ref{prop:smallLCD}}\eps_0 \rho'_j \sqrt{pn} \right) \le C_{\ref{prop:smallballprob}}^{n- \lceil \alpha n \rceil} \left( c_{\ref{prop:smallLCD}}\eps_0 + \frac{1}{\sqrt p L } \right)^{n- \lceil \alpha n \rceil}.
$$

Straightforward computations show
\begin{equation}\label{e:reqns}
\frac{c'_{\ref{prop:smallLCD}} \sqrt{\alpha n}}{L} \geq \frac{2C_{\ref{prop:smallballprob}}}{\sqrt{p}L}\quad\text{and}\quad \frac{c'_{\ref{prop:smallLCD}} \sqrt{\log r} (\log \log n)}{ \alpha^{2} r} \geq \frac{2C_{\ref{prop:smallballprob}}}{\sqrt{p} L}.
\end{equation}
Recall that $\eps_0$ as defined as the minimum of the two upper bounds in \eqref{e:reqns}, so
\begin{equation}\label{e:epszerolower}
\eps_0 \ge \frac{2 C_{\ref{prop:smallballprob}}}{\sqrt{p} L}.
\end{equation}
Then
$$
\P\left( \| (M_n - \lambda) x \|_2 \le  c_{\ref{prop:smallLCD}}\eps_0 \rho'_j \sqrt{pn} \right) \le C_{\ref{prop:smallballprob}}^{n- \lceil \alpha n \rceil} \left( c_{\ref{prop:smallLCD}}\eps_0 + (2 C_{\ref{prop:smallballprob}})^{-1} \eps_0 \right)^{n- \lceil \alpha n \rceil}.
$$
 Setting $c_{\ref{prop:smallLCD}} = (2C_{\ref{prop:smallballprob}})^{-1}$ and applying a union bound over all elements $x\in \mathcal M$, we obtain
\begin{equation}\label{e:p}
P \leq |\mathcal{M}| \eps_0^{n - \alpha n}.
\end{equation}


To bound $ |\mathcal{M}| \eps_0^{n - \alpha n}$ from \eqref{e:p}, we use \eqref{e:mbound11} and divide into two cases.  First, suppose $\frac{2\bar{c} L}{\sqrt{\alpha n}} \leq 1$.  
By \eqref{e:mbound11}, we have
\begin{equation}
|\mathcal{M}| \leq  \exp \left( n \left[ C + 3\alpha  \log(1/\eps_0)  
 \right] \right) \times    13^{n}.
\end{equation}
Combining this with \eqref{e:p} and absorbing the $13^{n}$ into the exponential yields 
$$
P \le \exp \left( n \left[ C + 3 \alpha  \log(1/\eps_0)  
 \right] \right) \times   \eps_0^{n - \alpha n},
$$
so
$$
P \le \exp \left( n \left[ C + 3\alpha \log(1/\eps_0)  - (1 - \alpha)\log(1/\eps_0)  
 \right] \right) \le  \exp(- c''_{\ref{prop:smallLCD}}n).
$$

In the last line we used $\alpha = o(1)$ and $\eps_0 \rightarrow 0$ (the latter is by direct calculation), so $\log(1/\eps_0) \rightarrow \infty$ and the term inside the brackets tends to $-\infty$.

For the case $\frac{2 \bar{c} L}{\sqrt{\alpha n}} > 1$, recalling the definition of $\eps_0$ and that $m = \alpha n$ gives 
\begin{equation}
P \le \exp \left( n \left[ C + 3 \alpha  \log(1/\eps_0)  
 \right] \right) \times   \left( \frac{13 \bar{c} L}{\sqrt{\alpha n}} \right)^{n-\alpha n} \left(\frac{c'_{\ref{prop:smallLCD}} \sqrt{\alpha n}}{L} \right)^{n - \alpha n},
\end{equation}
\begin{equation}\label{e:takesmall}
P \le \exp \left( n \left[ C + 3 \alpha  \log(1/\eps_0)  
 \right] \right) \times   \left( 13 \bar{c}c'_{\ref{prop:smallLCD}} \right)^{n-\alpha n} .
\end{equation}
Now \eqref{e:epszerolower} shows that 
$$
\frac{1}{\eps_0} \le \frac{\sqrt{p} L }{ 2 C_{\ref{prop:smallballprob}}}\le \frac{\exp(\alpha^{-1}) }{ 2 C_{\ref{prop:smallballprob}}}. 
$$
This, along with the stipulated range of $\alpha$, implies that 
$$
3 \alpha  \log(1/\eps_0) < C.
$$
Therefore, taking $c'_{\ref{prop:smallLCD}}$ small enough in \eqref{e:takesmall}, we have
$$ 
P \le \exp(- c''_{\ref{prop:smallLCD}}n).
$$
This completes Step 3.

We now proceed to Step 4. Having shown the result for all the points in the net, we now extend to the entire level set $\mathcal K$.  Again, we divide into cases.

We assume first that 
\begin{equation}\label{e:e0first}
\frac{c'_{\ref{prop:smallLCD}} \sqrt{\alpha n}}{L} \leq \frac{c'_{\ref{prop:smallLCD}} \sqrt{\log r} (\log \log n)}{ \alpha^2 r},\quad \text{so that}\quad\eps_0 = \frac{c'_{\ref{prop:smallLCD}} \sqrt{\alpha n}}{L}.
\end{equation}
  For any $w \in \mathcal K$, let $m \in \mathcal{M}$ be the closest element of the net $\mathcal M$. Then, by the definition of $\mathcal M$,
\begin{align*}
\|w - m\|_2 &\leq \frac{c_{\ref{prop:smallLCD}}\rho'_j \eps_0}{10K} + \sum_{k=1}^{k_0} \left(\Big \| w_{I_k} - \|w_{I_k}\|_2 y_k \Big \|_2 + \Big \| \|w_{I_k}\|_2 y_k - t_k y_k\Big\|_2 \right) \\
&\leq \frac{c_{\ref{prop:smallLCD}}\rho'_j \eps_0}{10K} + \sum_{k=1}^{k_0} \left( \left\| \frac{w_{I_k}}{\|w_{I_k}\|_2} - y_k \right \|_2 \|w_{I_k}\|_2 +  \Big \| \|w_{I_k}\|_2 y_k - t_k y_k\Big\|_2 \right) \\
&\leq \frac{c_{\ref{prop:smallLCD}}\rho'_j \eps_0}{10K} +  k_0 \rho'_j \alpha^{-1} (c'_{\ref{prop:compressible}})^{-2} 2^3 \beta + k_0 \frac{c'_{\ref{prop:smallLCD}}  \rho'_j\eps_0}{10 K k_0}  \\
&\leq \frac{c_{\ref{prop:smallLCD}}\rho'_j \eps_0}{5K} +   
C_\gamma  ( c'_{\ref{prop:compressible}})^{-2}   \frac{\rho'_j}{\alpha^{5/2} L \sqrt{p}} \\
&\leq \frac{c_{\ref{prop:smallLCD}}\rho'_j \eps_0}{2K}.
\end{align*}
In the third inequality, we used that there are $k_0$ terms in the sum, that the $y_k$ form a $\beta$-net, and the upper bound on $\| w_k \|_2$ from \eqref{eq:l2bound}.
In the fourth inequality, we used $k_0 \le \alpha^{-1}$ from \eqref{e:kbounds} and the inequality
\begin{equation}\label{e:betabound}
 \beta \le \frac{C_\gamma \alpha^{-1/2} }{L \sqrt{p} },
\end{equation}
where $C_\gamma$ is a constant that depends only on $\gamma$. The inequality \eqref{e:betabound} follows from the definition of $\beta$ and the hypothesized upper bound $\log(\sqrt{p}L ) \le \alpha^{-1}$ on $L$. The last inequality follows by direct calculation using the value of $\eps_0$ given in \eqref{e:e0first} and the assumed lower bound on $\alpha$.

For the other case, suppose
\begin{equation}\label{e:otherecase}
\frac{c'_{\ref{prop:smallLCD}} \sqrt{\alpha n}}{L} \ge \frac{c'_{\ref{prop:smallLCD}} \sqrt{\log r}(\log \log n)}{ \alpha^2 r}  ,\quad \text{so that}\quad\eps_0 = \frac{c'_{\ref{prop:smallLCD}} \sqrt{\log r}(\log \log n)}{ \alpha^2 r} .
\end{equation}
For any $w \in \mathcal K$, let $m \in \mathcal{M}$ be the closest element of the net $\mathcal M$. Then, by the definition of $\mathcal M$,
\begin{align*}
\|w - m\|_2 &\leq \frac{c_{\ref{prop:smallLCD}} \rho'_j \eps_0}{10K} + \sum_{k=1}^{k_0} \left(\Big \| w_{I_k} - \|w_{I_k}\|_2 y_k \Big \|_2 + \Big \| \|w_{I_k}\|_2 y_k - t_k y_k\Big\|_2 \right) \\
&\leq \frac{c_{\ref{prop:smallLCD}} \rho'_j \eps_0}{10K} + \sum_{k=1}^{k_0} \left( \left\| \frac{w_{I_k}}{\|w_{I_k}\|_2} - y_k \right \|_2 \|w_{I_k}\|_2 +  \Big \| \|w_{I_k}\|_2 y_k - t_k y_k\Big\|_2 \right) \\
&\leq \frac{c_{\ref{prop:smallLCD}} \rho'_j \eps_0}{10K} +  k_0 \rho'_j \alpha^{-1} (c'_{\ref{prop:compressible}})^{-2} 2^3 \beta  + k _0 \frac{c_{\ref{prop:smallLCD}} \rho'_j\eps_0 }{10 K k_0}  \\
&\leq \frac{c_{\ref{prop:smallLCD}} \rho'_j \eps_0}{5K} + k_0 \rho'_j \alpha^{-1} (c'_{\ref{prop:compressible}})^{-2}   2^3 \frac{ \sqrt{\log(4 \sqrt{\gamma p} L)}}{\sqrt{\gamma p}L }  \\
&\leq \frac{c_{\ref{prop:smallLCD}} \rho'_j \eps_0}{5K} +  \rho'_j \alpha^{-2} (c'_{\ref{prop:compressible}})^{-2} 2^4 \frac{ \sqrt{\log(4 \sqrt{\gamma} r)}}{ \sqrt{\gamma }r }    \\
&\leq \frac{c_{\ref{prop:smallLCD}} \rho'_j \eps_0}{5K} +   \left(  (c'_{\ref{prop:compressible}})^{-2} (c'_{\ref{prop:smallLCD}})^{-1} \gamma^{-1/2}2^5  \right)  \rho'_j \eps_0  (\log \log n)^{-1} \\
&\leq \frac{c_{\ref{prop:smallLCD}} \rho'_j \eps_0}{2K}. 
\end{align*}
In the third line, we used that there are $k_0$ terms in the sum, that the $y_k$'s form a $\beta$-net, and the upper bound on $\| w_k \|_2$ from \eqref{eq:l2bound}. The fourth line follows from the definition of $\beta$.
The fifth line is a result of the observation that $\sqrt{\log x}/x$ is a decreasing function for large $x$, $r \rightarrow \infty$, and  $r < L\sqrt{p}$.  We also used the bound $k_0 \le \alpha^{-1}$ from \eqref{e:kbounds}.
In the sixth line, we used the definition of $\eps_0$ in \eqref{e:otherecase}.
For the the last line, we used $ (\log \log n )^{-1}= o(1)$ and took $n$ large enough.

Therefore, if $\|(M_n - \lambda) w\|_2 \geq 2{c_{\ref{prop:smallLCD}}}\eps_0 \sqrt{pn}$, then using Lemma~\ref{l:opnorm},
$$
\|(M_n - \lambda) m\|_2 \geq 2 c_{\ref{prop:smallLCD}}\eps_0  \sqrt{pn} - \|M_n - \lambda\| \frac{c_{\ref{prop:smallLCD}} \eps_0}{2K} \geq c_{\ref{prop:smallLCD}} \eps_0  \sqrt{pn},
$$
with exponentially small error probability, which contradicts the conclusion of Step 3 above. After adjusting $c_{\ref{prop:smallLCD}}$ by a factor of $2$, this completes the proof.
\end{proof}
\begin{remark}
As noted in Remark~\ref{r:optimal}, the optimal result should permit $p$ as small as $C \log(n) /n$.  The restriction that $p \ge C \log^{7 + \nu} n /n$ in the above proof comes from the requirement that $\eps_0 \rightarrow 0$. \end{remark}

We now extend the previous result to all vectors with small LCD.
\begin{proposition} \label{p:smallLCDallLevels} 
Fix $\nu >0$. There exists a constant $c_{\ref{p:smallLCDallLevels}} >0$ such that for $p \geq C_{\ref{prop:smallLCD}} \frac{\log^{7 + \nu} n}{n}$, $\lambda \in [-K \sqrt{pn}, K \sqrt{pn}]$, $j \in \N$ and for any
$$
(np)^{-1/(7 + \nu)} \leq \alpha \leq \frac{c'_{\ref{prop:smallLCD}}}{\log n}
$$
the following holds. The probability that there exists $v \in \Incomp(\alpha n , \rho)$ such that 
$$
 \| (M_n -\lambda)v\|_2 \leq c_{\ref{prop:smallLCD}}\eps_1 \rho'_1 \sqrt{pn} \text{ and } \widehat D(v) \leq p^{-1/2} \exp(\alpha^{-1})
$$
is at most $\exp(-c_{\ref{p:smallLCDallLevels}} n)$ for $n \ge (c_{\ref{p:smallLCDallLevels}} )^{-1}$, where 
$$
\eps_1 =  \exp ( -  c_{\ref{p:smallLCDallLevels}} n^{1/7} ) .
$$
\end{proposition} 

\begin{proof}
We set $D_0 = c'_{\ref{prop:compressible}} 2^{-5} \alpha^{3/2} n^{1/2}$ and recall that $\widehat{D}(v) \ge D_0$ by $\eqref{e:Llower}$. We can decompose the relevant vectors as
$$
\bigcup_{j' = 0 }^{\log_2 \left(p^{-1/2} \exp(\alpha^{-1})\right)}
\bigcup_{j=0 }^{ \log_2 \rho^{-1}}   \left(\Incomp(m, 2^{j} \rho, 2^{j+1}\rho) \cap S_{2^{j'} D_0} \right),
$$
\noindent where we used $D_0 \ge 1$. Recall $\log(1/\rho) = O(\log n/ \log \log n)$ by Remark~\ref{r:rhoremark}.
Similarly, the number of $j'$ indices in the union is $O(\log n)$ because each of $\log_2 p^{-1/2}$ and  $\log_2 \exp(\alpha^{-1})$ are $O(\log n)$.
 Therefore, taking a union bound, applying Proposition \ref{prop:smallLCD}, and observing $\rho'_j \geq \rho'_1$ and $\eps_0(L) \ge \eps_1$ for the $\eps_0(L)$ defined in Proposition~\ref{prop:smallLCD} yields the result. \end{proof} 

\subsection{Eigenvector Bound}\label{s:evb}
We now come to a key proposition used in the proof of the main theorem. 
\begin{proposition} \label{p:eigvectors}
For $M_n$ as in Definition~\ref{def:matrixmodel}, there exists a constant $c_{\ref{p:eigvectors}} > 0$ such that for 
$$
(np)^{-1/(7 + \nu)} \leq \alpha \leq \frac{c'_{\ref{prop:smallLCD}}}{\log n},
$$
the probability that $M_n$ has an eigenvector  v  such that 
\begin{equation*}v \notin \Comp(\alpha n, \rho) \cup \Dom(\alpha n, c_{\ref{prop:compressible}} )\text{ and } \widehat D(v ) \leq p^{-1/2} \exp(\alpha^{-1})\end{equation*} is at most $\exp(-c_{\ref{p:eigvectors}} n)$, for $n \ge (c_{\ref{p:eigvectors}})^{-1}$.
\end{proposition}
\begin{proof}
Consider a $c_{\ref{prop:smallLCD}}\eps_1 \rho'_1 \sqrt{pn}$-net of $[-K \sqrt{pn}, K\sqrt{pn}]$, where $\eps_1$ was defined in Proposition~\ref{p:smallLCDallLevels}. For an eigenvalue $\lambda \in [-K \sqrt{pn}, K \sqrt{pn}]$, there exists a point of the net $\lambda_0$ such that for corresponding eigenvector $v$ we have
$$
\|(M_n - \lambda_0) v\|_2 = |\lambda - \lambda_0| \leq c_{\ref{prop:smallLCD}} \eps_1 \rho'_1 \sqrt{pn}.
$$
However, by a union bound and Proposition~\ref{p:smallLCDallLevels}, the probability of this event is bounded by $\exp(-c_{\ref{p:eigvectors}} n)$ for some $c_{\ref{p:eigvectors}}>0$.  By Lemma~\ref{l:opnorm}, decreasing the value of $c_{\ref{p:eigvectors}}$ can account for the event that there exists an eigenvalue of $M_n$ outside the interval $[-K \sqrt{pn}, K\sqrt{pn}]$. This concludes the proof.
\end{proof}

\section{Proofs of Main Results}
\subsection{Proof of Theorem \ref{thm:main}} \label{sec:mainproof}
In preparation for the main proof, we record the following lemma from \cite{luh2018sparse}.
\begin{lemma}[{\cite[Lemma 6.1]{luh2018sparse}}]\label{l:largecoordinates}
For any $v \in \Incomp(m,\rho)$, 
$$
\left| \left\{ i : \frac{\rho^2}{\sqrt{2n}} \le |v_i | \le \frac{1}{\sqrt{m}}  \right\}  \right| \ge \frac{ m \rho^2}{2}.
$$
\end{lemma}

\begin{proof}[Proof of Theorem \ref{thm:main}]
We repeat the decomposition described in Section \ref{sec:proofstrategy}.  Let
\begin{equation} \label{eq:matrixdecomp}
M_n = 
\begin{pmatrix} 
  M_{n-1} & X \\ 
  X^T & m_{nn} 
\end{pmatrix},
\end{equation}
where $X = (x_1, \dots, x_{n-1}) \in \mathbb{R}^{n-1}$.  Let $v = (x , a)$ (where $x \in \mathbb{R}^{n-1}$ and ${a \in \mathbb{R}}$) be the unit eigenvector associated to $\lambda_i(M_n)$. Because $v$ is an eigenvector with eigenvalue $\lambda_i$,
$$
\begin{pmatrix}
M_{n-1} & X \\ 
X^T & m_{nn}
\end{pmatrix}
\begin{pmatrix}
x \\
a
\end{pmatrix} = 
\lambda_i(M_n) \begin{pmatrix}
x \\
a
\end{pmatrix}.
$$
Considering the top $n-1$ coordinates gives
$$
(M_{n-1} - \lambda_i(M_n)) x + a X = 0.
$$
Let $w$ be the eigenvector of $M_{n-1}$ corresponding to $\lambda_{i}(M_{n-1})$.  After multiplying on the left by $w^T$, we arrive at
\begin{equation}
|a w^T X | = |w^T (M_{n-1} - \lambda_i(M_n)) x| = |\lambda_i(M_{n-1}) - \lambda_i(M_n)||w^T x|.
\end{equation}
Since $|w^T x | \le 1$ by the Cauchy--Schwarz inequality, this implies
\begin{equation}\label{eq:eiggap}
|w^T X |  \le \frac{1}{|a|} |\lambda_i(M_{n-1}) - \lambda_i(M_n)|.
\end{equation}

By the Cauchy interlacing law, we must have $\lambda_i(M_n) \leq \lambda_i (M_{n-1}) \leq \lambda_{i-1} (M_n)$. For any $\hat \delta >0$, let $\mathcal{E}_i = \mathcal{E}_i \left( \hat \delta \right)$ denote the event that 
\begin{equation}\label{e:gggg}
\lambda_{i+1} - \lambda_i  \leq \hat \delta  \sqrt{\frac{p}{n}}.
\end{equation}
On $\mathcal E_i$, \eqref{eq:eiggap} implies
\begin{equation}\label{e:123}
|w^T X |   \le \hat \delta  \sqrt{\frac{p}{n}} \frac{1}{|a|}.
\end{equation}

Now note that the decomposition (\ref{eq:matrixdecomp}) can be done along any coordinate, not just the last.  For any $A>0$, let $n_A$ be the number of coordinates with absolute value at least $A$, and let $N$ be a parameter. Therefore, repeating the argument leading to \eqref{e:123} with the coordinate $a$ chosen uniformly at random, and considering the probability that we choose a coordinate with absolute value at least $A$, and $\mathcal E_i$ obtains, we find
\begin{align}
\P(\mathcal{E}_i) =&
\P(\mathcal{E}_i \cap \{  n_A \ge N \} ) + \P(\mathcal{E}_i \cap \{  n_A < N \} ) \\
 \leq& \frac{n}{N} \P \left( |w^T X| \leq \hat \delta \sqrt{\frac{p}{n}} \frac{1}{A} \right) + \P(n_A < N).\label{e:66}
\end{align}

Setting $m = c_{\ref{prop:compressible}} n $ in Proposition \ref{prop:eigvecnotcomp} shows that any eigenvector $v$ will not be in $\Comp(c_{\ref{prop:compressible}}n, \rho)$ with exponentially high probability. When $v \notin \Comp(c_{\ref{prop:compressible}}n, \rho)$, by Lemma~\ref{l:largecoordinates}, there are greater than $c_{\ref{prop:compressible}} n \rho^2/2$ coordinates whose absolute values are larger than $\rho/\sqrt{2n}$.  We set $N =c_{\ref{prop:compressible}} n \rho^2/2$ and $A = \rho/\sqrt{2n}$ in \eqref{e:66} to find
\begin{equation}\label{e:controlme}
\P(\mathcal{E}_i) \leq \frac{2}{c_{\ref{prop:compressible}} \rho^2} \P \left( |w^T X| \leq \hat \delta \rho^{-1} \sqrt{2p}   \right) + \exp(- c_{\ref{prop:eigvecnotcomp}} p n).
\end{equation}

With probability at least $1 - \exp(- c_{\ref{p:eigvectors}} p n)$, $$\widehat D(w) \geq p^{-1/2} \exp(\alpha^{-1})$$ by Proposition \ref{prop:eigvecnotcomp} (applied with $m=\alpha n$) and Proposition~\ref{p:eigvectors}.  
At this point, we would like to apply Proposition \ref{prop:smallballprobability}
to control the probability $\P \left( |w^T X| \leq \hat \delta \rho^{-1} \sqrt{2p}   \right)$ in 
\eqref{e:controlme}.
However, this proposition applies to the LCD $D(w)$, not the regularized LCD $\widehat D(w)$, so a slightly more delicate argument is required.

By the definition of regularized LCD, there exists some subset $J$ of coordinate indices such that
$$ D\left( \frac{w_J }{\| w_J \|_2}  \right)\ge p^{-1/2} \exp(\alpha^{-1}).$$ To adjust for the regularized LCD, we observe that conditioning on a subset of $X$ can only increase the L\'evy function $\mathcal L ( w^T X, \eps)$ for any $\eps >0$.  We condition on all the random variables in $X$ whose indices do not lie in the subset $J$.  Also, to apply Proposition \ref{prop:smallballprobability}, we need to normalize this subset to be on the unit sphere.   
Therefore, by Proposition \ref{prop:smallballprobability},
\begin{equation}\label{e:thrf}
\mathcal{L}(w^T X, \hat \delta \rho^{-1} \sqrt{2p}) \leq \mathcal{L}\left(\frac{w_J^T}{\|w_J\|_2} X_J,  \frac{\hat \delta \rho^{-1} \sqrt{2p}}{\|w_J\|_2} \right) \leq \frac{ 2 \sqrt{2} C_{\ref{prop:smallballprobability}} \hat \delta  \rho^{-1}}{\|w_J\|_2},
\end{equation}
for all $\hat \delta \geq \rho e^{- \alpha^{-1} }/\sqrt{2}$. 
By Lemma \ref{ikbounds}, $\|w_J\|_2 \geq c'_{\ref{prop:compressible}} 2^{-3} \rho \alpha$. 
Therefore, putting \eqref{e:thrf} into \eqref{e:controlme}, we find

\begin{equation}\label{e:deltadef}
\P(\mathcal{E}_i) \leq \frac{32 \sqrt{2}}{c_{\ref{prop:compressible}} c'_{\ref{prop:compressible}}\rho^4   \alpha}  C_{\ref{prop:smallballprobability}} \hat \delta   + \exp(- c_{\ref{prop:eigvecnotcomp}} p n).
\end{equation}
 We set $\delta  =  \hat \delta \rho^{-4}$. Then the above holds for $\delta \ge \rho^{-3} e^{- \alpha^{-1} }/\sqrt{2}$.  Recall that $\rho^{-3} = \exp(O(\log n/ \log \log n))$.  Thus, we obtain the theorem after lowering $c'_{\ref{thm:main}}$, which constrains the range of $\alpha$.
\end{proof}

\subsection{Proof of Theorem \ref{thm:maingraph}} \label{sec:randomgraph}
Let $G(n,p)$ denote the Erd\H{o}s--R\'{e}nyi random graph on $n$ vertices with edge probability $p$, and let $A_n$ denote the adjacency matrix of $G(n,p)$. In other words, $A_n$ is a symmetric matrix of Bernoulli variables with parameter $p$, with all $0$ entries on the diagonal. We have $\E A_n = p (J_n - I_n)$ where $J_n$ is the matrix of all ones, so our main theorem does not apply. However, only small modifications are necessary to handle this case, which we detail in this section, following closely the analogous argument in \cite[Section 8]{luh2018sparse}. 

First, we observe that Proposition~\ref{prop:compressible} can be adapted so that the proposition holds for $A_n$ in place of $M_n$. This was proved in \cite[Appendix B]{luh2018sparse}. It follows that Proposition~\ref{prop:eigvecnotcomp} also holds for $A_n$ (by repeating the proof of Proposition~\ref{prop:eigvecnotcomp} using the analogue of Proposition~\ref{prop:compressible} for $A_n$).

Next, we claim that Proposition~\ref{p:smallLCDallLevels} can be adapted to hold for the matrix {$A_n - p(J_n -I_n)$} in place of $M_n$, \emph{with the additional restriction that we must suppose $p \le 1/2$}. The restriction is due to the fact that we will write the off-diagonal entries of this matrix as $a_{ij} =\delta_{ij} \xi_{ij}$, where $\delta$ is Bernoulli with parameter $2p$ and $\xi_{ij}$ is Bernoulli with parameter $1/2$ (as in the definition of $M_n$).
Our arguments for Proposition~\ref{p:smallLCDallLevels} revolved around L\'{e}vy concentration and nets. The use of L\'evy concentration in Proposition~\ref{prop:smallballprob} does not need to be modified for the random graph case, since it is invariant under changes in the mean of the matrix.\footnote{However, it does require the aforementioned decomposition $a_{ij} =\delta_{ij} \xi_{ij}$, giving rise to the $p \le 1/2$ restriction.} For the nets, we required the operator norm bound Lemma~\eqref{l:opnorm}; we claim the analogue of this statement for {$A_n - p(J_n -I_n)$} also holds. A straightforward modification of the proof of \cite[Theorem 1.7]{basak2017sparse} shows
\begin{equation}
\P( \| A_n - p ( J_n - I_n) \|_2 \ge K' \sqrt{pn} ) \le \exp(- c' pn)
\end{equation}
for some $K', c'>0$.
We obtain that Proposition~\ref{p:smallLCDallLevels} holds for $A_n - p(J_n-I_n)$, if $p \le 1/2$. 

Additionally, we need a slight generalization of Proposition~\ref{p:smallLCDallLevels},
 which lower bounds not just $\| (A_n  - p( J_n - I_n)  - \lambda)v  \|_2$, but 
 \begin{equation}\label{e:generalize1}
 \| (A_n  - p( J_n - I_n)  - \lambda)v - x \|_2
 \end{equation}
 for any fixed vector $x$. This generalization holds because the high probability lower bounds used to prove Proposition~\ref{p:smallLCDallLevels} come from Proposition~\ref{prop:smallballprob}, and the latter proposition concerns L\'evy concentration, which is by definition translation invariant.
 

We now turn to the proof of Theorem \ref{thm:maingraph}.
\begin{proof}[Proof of Theorem  \ref{thm:maingraph}] Above, we established that the analogue of Proposition~\ref{p:smallLCDallLevels} holds for $A_n - p(J_n-I_n)$, if $p \le 1/2$. This restriction motivates the following division into cases.

{\bf Case I: $p \le 1/2$.} Our preliminary goal to is establish that Proposition~\ref{p:eigvectors} holds for $A_n$. We have
\begin{equation}
\{ J_n x \colon x \in \S^{n-1} \} = \{ \theta \cdot \one \colon \theta \in [-n , n]\}
\end{equation}
where $\one$ is the vector $(1,\dots, 1)$ of all ones. Set $\mathcal X_n = \{ \kappa \cdot \one \colon \kappa \in [ -pn , pn] \}$.  Let $\mathcal B$ be a $c_{\ref{p:smallLCDallLevels}} \varepsilon_0 \rho' \sqrt{pn} $-net of $\mathcal X_n$ such that
\begin{equation}\label{e:BBB}
| \mathcal B | \le \frac{4pn}{c_{\ref{p:smallLCDallLevels}} \varepsilon_0 \rho' \sqrt{pn}}
\le C \exp(c n^{1/7}).
\end{equation}
For $x, x' \in \mathcal  X_n$, the reverse triangle inequality yields
\begin{equation}\label{e:rti}
\left|   \| (A_n  - p( J_n - I_n)  - \lambda)v - x \|_2 - \| (A_n  - p( J_n - I_n)  - \lambda)v - x' \|_2   \right| \le \| x - x' \|_2,
\end{equation}
so any $(A_n  - p( J_n - I_n)  - \lambda)v - y$ with $y \in \mathcal X_n$ can be well approximated by $(A_n  - p( J_n - I_n)  - \lambda)v - x $ for some $x \in \mathcal B$.

Define 
$$S_D = \left\{ v  \in  \Incomp(\alpha n , \rho): \widehat D(v) \le p^{-1/2} \exp(\alpha^{-1}) \right\}.$$
By \eqref{e:rti}, a union bound over the net $\mathcal B$, and the  analogue of Proposition~\ref{p:smallLCDallLevels} for \eqref{e:generalize1} stated above, we obtain
\begin{equation}\label{e:daunionb}
\P( \inf_{x\in \mathcal X_n} \inf_{v\in S_D} \| (A_n  - p( J_n - I_n)  - \lambda)v - x  \|_2 \le c_{\ref{p:smallLCDallLevels}} \varepsilon_0 \rho'_1 \sqrt{pn} ) \le \exp( - cn )
\end{equation}
for any single $\lambda \in [-K' \sqrt{pn}, K' \sqrt{pn}]$.
 After observing that 
\begin{equation}
\inf_{x\in \mathcal X_n} \inf_{v\in S_D} \| (A_n  - p( J_n - I_n)  - \lambda)v - x  \|_2 \le 
\inf_{v\in S_D} \| ( A_n - (\lambda - p) ) v \|_2,
\end{equation}
we find 
\begin{equation}\label{ee:notcomp}
\P(\inf_{v\in S_D} \| ( A_n - (\lambda - p) ) v \|_2 \le c_{\ref{p:smallLCDallLevels}} \varepsilon_0 \rho'_1 \sqrt{pn} ) \le \exp( - cn ).
\end{equation}

Using \eqref{ee:notcomp} in place of  Proposition~\ref{p:smallLCDallLevels} in the proof of Proposition~\ref{p:eigvectors}, we find that Proposition~\ref{p:eigvectors} holds for $A_n$ in place of $M_n$.  

We can now repeat the proof of Theorem~\ref{thm:main} to prove theorem in this case, with the appropriate analogues for $A_n$ substituting for Proposition~\ref{p:eigvectors} and Proposition~\ref{prop:eigvecnotcomp}. (The latter was noted at the beginning of Section~\ref{sec:randomgraph}.)

{\bf Case II: $p  > 1/2$.} Observe that the adjacency matrix $A_n(p)$ of $G(n,p)$ is equal in distribution to $J_n - I_n - A_n(1-p)$.  Hence controlling $$\| (A_n(p) - p(J_n -I_n)  - \lambda ) v\|_2$$ is equivalent to controlling $$\| (A_n(1-p) - (1-p) (J_n - I_n)  + \lambda) v\|_2.$$ This reduces the problem to Case I and completes the proof.
\end{proof}
\begin{remark}\label{r:finite}
The size of the one-dimensional net $\mathcal B$ in \eqref{e:BBB} is compensated by the $\exp(-cn)$ error probability used for the union bound in \eqref{e:daunionb}. For general finite-rank perturbations by a finite linear combination of matrices of the form $n  \cdot v v^T$ for $ v \in \S^n$, one simply adds more one-dimensional nets and completes the argument in the same way. However, for perturbations whose rank grows even moderately quickly, the combined size of the necessary supplemental nets becomes too large.
\end{remark}

\subsection{Proof of Theorem \ref{thm:nodaldomains}} \label{sec:application}

The following is essentially Lemma 9.1 of \cite{nguyen2017gaps}. We provide the proof for completeness.

\begin{lemma}\label{l:91}
For any $A>0$ there exists $B=B(A)>0$ such that the following holds with probably at least $1 - O(n^{-A})$. If there exist $\lambda \in \R$ and $v\in \S^{n-1}$ such that $\| (A_n   - \lambda) v \| \le n^{-B}$, then $A_n$ has an eigenvector $u_{i_0} \in \S^{n-1}$ and corresponding eigenvalue $\lambda_{i_0}$ such that $$| \lambda_{i_0} - \lambda|  < n^{-B/4}\quad\text{and}\quad\| v - u_{i_0} \| < n^{ - B/4}.$$ 
\end{lemma}

\begin{proof} From our main result, Theorem~\ref{thm:maingraph}, we may suppose that all eigenvalue gaps satisfy $| \lambda_j - \lambda_i| \ge n^{-B/2}$. Let $v = \sum c_i u_i$ express $v$ as a linear combination of unit eigenvectors of $A$. There must exist $i_0$ such that $c_{i_0} \ge n^{-1/2}$. So 
\begin{equation}
\| (A _n  - \lambda) v \|  = \left( \sum_{i=1}^n c_i^2 (\lambda_i - \lambda)^2 \right)^{1/2} 
\end{equation}
implies, assuming $\| (A_n   - \lambda) v \| \le n^{-B}$, that $|\lambda - \lambda_{i_0} | \le n^{-B + 1/2}$. This implies the first conclusion. Then because all gaps satisfy $| \lambda_j - \lambda_i| \ge n^{-B/2}$ we have that  $|\lambda - \lambda_{i} | \ge n^{-B/2}/2$ for all $i\neq i_0$. But then we must have $| c_i | = O( n^{-B/2})$ for $i\neq i_0$, implying the second conclusion. \end{proof}

\begin{proof}[Proof of Theorem \ref{thm:nodaldomains}]  We follow the proof of Theorem 3.3 in \cite{nguyen2017gaps}. After adjusting $C$ by adding $1$, it suffices to prove the claim for a single coordinate and use a union bound. Write $A=A_n$ and let its first column be $(a_{11}, X)$ where $X$ is a vector of $n-1$ coordinates. Let $v = (v_1, v')$ be an eigenvector with eigenvalue $\lambda$ so that
\begin{equation}
 v_1 m_{11} +(v')^T X = \lambda v_1, \quad (A _{n-1} - \lambda) v' = -v _1 X.
\end{equation}
Suppose that $|v_1 | \le n^{-D}$ where $D$ will be chosen later. By taking $D$ large enough, using that the entries of $A$ are bounded, and adding $O(N^{-D})$ mass to the first component of $v'$ to make it unit norm, it suffices to show that 
\begin{equation}
\| (A _{n-1} - \lambda) v' \| \le n^{-D/2}\quad \text{and}\quad|(v')^T X| \le n^{-D/2}
\end{equation}
occur jointly with low probability. By Lemma~\ref{l:91}, if the first condition holds then there exists an eigenvector $u'$ of $A_{n-1}$ with $\| u'  - v'\|_2 \le n^{-D/8}$. Then  $|(v')^T X| \le N^{-D/2}$ implies  $|(u')^T X| \le n^{-D/16}$. We claim this contradicts a statement established in the proof of Theorem~\ref{thm:main}.

In \eqref{e:thrf} and the following lines, we showed 
\begin{equation}
\P \left( |w^T X| \leq  \delta \rho^3 \sqrt{p} \right) \le C \frac{\rho^2}{\alpha} \delta,
\end{equation}
where $\delta$ was defined below \eqref{e:deltadef} (in terms of $\hat \delta$). Now we take $\delta = n^{-D/16}/\rho^3 \sqrt{p}$, $\alpha = (np)^{-1/(7 + \nu)}$ and $p > C \log^{7 + \nu} (n)/n$, which proves the theorem after taking $D$ large enough. \end{proof}

\bibliographystyle{plain}
\bibliography{tail}
\end{document}